\documentclass [11pt] {article}

\usepackage {amssymb,latexsym}
\usepackage {amsmath}
\usepackage{amsthm}
\usepackage{tikz-cd}

\title {Free Reductive Lie Algebra Pairs of Lie-Yamaguti algebras}

\author {
	{\bf Sa\"{\i}d Benayadi
		\thanks{said.benayadi@univ-lorraine.fr}~,
		\addtocounter{footnote}{1}}\\[3mm]
Universit\'{e} de Lorraine, Laboratoire IECL, 
	CNRS-UMR 7502, UFR MIM,\\ 3 rue Augustin Frenel,
	 BP 45112, 57073 Metz Cedex 03, France \\[3mm]
	     {\bf
          Martin
          Bordemann\thanks{martin.bordemann@uha.fr}~,
          \addtocounter{footnote}{2}} \\[3mm]
         D\'{e}partement de Math\'{e}matiques, Labo IRIMAS, 
         Universit\'{e} de Haute Alsace, France\\[3mm]
         {\bf Friedrich Wagemann\thanks{Friedrich.Wagemann@univ-nantes.fr}~,
         \addtocounter{footnote}{3}}\\[3mm]
         Laboratoire de math\'{e}matiques Jean Leray
         UMR 6629 du CNRS,
         Universit\'{e} de Nantes, France  
}             

\usepackage{graphicx}

\newtheorem {lemma} {Lemma}
\newtheorem {proposition} [lemma] {Proposition}
\newtheorem {theorem} [lemma] {Theorem}
\newtheorem {corollary} [lemma] {Corollary}
\newtheorem {definition}[lemma] {Definition}
\theoremstyle{definition}

\newtheorem{example}[lemma]{Example}

\makeatletter
\@addtoreset{equation}{section}

\makeatother

\begin{document}

\maketitle
\thispagestyle{empty}

\begin{abstract}
The goal of this article is to show the categorical links between on the one
hand the category of reductive Lie algebra pairs $\mathcal{RLP}$ and on the
other hand the category of Lie-Yamaguti algebras $\mathcal{LY}$. The fact that
the well-known construction of an enveloping algebra associating to a
Lie-Yamaguti algebra a reductive Lie algebra pair is not functorial leads us
to the main construction of the
article, namely a left adjoint to the natural restriction functor
$G:\mathcal{RLP}\to\mathcal{LY}$. As a final result we observe that the
construction of the enveloping algebra becomes functorial when one restricts
the morphisms of the categories $\mathcal{RLP}$ and $\mathcal{LY}$ to the
surjective ones. Then it becomes a right adjoint to the restriction functor.
\end{abstract}

\noindent {\it Keywords:} Lie algebras, Lie-Yamaguti algebra ,  reductive Lie algebra pair,  homogeneous spaces, category theory, Lie triple systems, Leibniz algebras.

\noindent {\it MSC:}   17A30, 17A32, 17A40, 18A40, 17B05, 22F30, 53B05.


\newpage

%
%
~~~\hfill \textit{Dedicated to the memory of Otto H.~Kegel.}
\section*{Introduction}
 \addcontentsline{toc}{section}{Introduction}
 
Otto H.~Kegel has written most of his
papers in group theory, but he was also quite open to other algebraic subjects 
like for instance associative algebras and Lie algebras: during the Master thesis of one of the authors, M.B., about certain Lie algebras, he used to say that established properties of groups can very often be copied to other fields.\\
 In this note, we shall be dealing with a categorical relation between the class of \textit{reductive Lie algebra pairs} $\mathcal{RLP}$ and the class of \textit{Lie-Yamaguti algebras} $\mathcal{LY}$. The first class comes from a Lie 
theoretical study of reductive homogeneous spaces (mostly in the 1950's, see e.g.~\cite{Nom54}, \cite[p.~190]{KN69}), and consists of triples $\big(\mathfrak{g},\mathfrak{h},\mathfrak{m}\big)$ consisting of a Lie algebra
$\mathfrak{g}$, a subalgebra $\mathfrak{h}\subset\mathfrak{g}$, and an 
$\mathrm{ad}_\mathfrak{h}$ invariant subspace $\mathfrak{m}\subset\mathfrak{g}$ such that
 $\mathfrak{g}=\mathfrak{h}\oplus \mathfrak{m}$: note that in that definition `reductive' does \textbf{not}
 mean that one of the two Lie algebras  $\mathfrak{g}$ or $\mathfrak{h}$
 is a reductive Lie algbra in the sense that its adjoint representation is completely reducible.  The second class is  K.~Yamaguti's generalization (see \cite{Yam58}) of N.~Jacobson's Lie triple systems \cite{Jac49} to a vector space (or even module) $E$ (interpreted as the tangent space at the distinguished point of the homogeneous space) equipped with a bilinear operation $T$ (representing the differential geometric torsion) and a trilinear operation $R$ (representing the differential geometric curvature) of the canonical connection on a reductive homogeneous space, see e.g.~\cite[p.~193]{KN69}. Both operations $T$ and $R$ satisfy six identities derived
from the classical Bianchi identities, see e.g.~\cite[p.~135]{KN63}. This latter approach can be seen as an redundancy-free description of the
affine geometry of the homogeneous space, but is computationally more involved.
Both classes $\mathcal{RLP}$ and $\mathcal{LY}$ are categories whose morphisms
are Lie algebra morphisms preserving the splitting and linear maps intertwining
the bilinear and trilinear operations, respectively.
Lie-Yamaguti algebras have been studied in numerous works more recently, see \cite{CD}
(from the operadic point of view) and 
\cite{BEM1}, and \cite{BEM2} for simple algebras in relation to Jordan algebras.
The aim of this note is to study and establish a \textit{functorial relation between the two categories}
$\mathcal{RLP}$ and $\mathcal{LY}$: such a link --controlled by category theory--
can be quite advantageous in order to translate or copy properties or concepts between the two, and the recent operadic approaches to classes of nonassociative
algebras, see e.g. \cite{CD}, are formulated in this language.\\
First, there is a well-known `easy' functor 
$\mathbf{G}':\mathcal{RLP}\to \mathcal{LY}$ sending an object $\big(\mathfrak{g},\mathfrak{h},\mathfrak{m}\big)$ in $\mathcal{RLP}$ to the complement $\mathfrak{m}$ and projecting (iterated) Lie brackets to the subspace and the subalgebra to obtain the bilinear and trilinear operations of a Lie Yamaguti algebra.\\
Secondly, it turns out to be much more difficult to find a functor $\mathbf{F'}$ in the other direction $\mathcal{LY}\to \mathcal{RLP}$:  the problem we faced is the fact
that to each Lie-Yamaguti algebra $E$ one can assign a well-known classical
reductive Lie algebra pair, its so-called \textit{enveloping Lie algebra}, see e.g.~\cite{Yam58} or \cite[p.~542]{KW01}, $\hat{\mathfrak{g}}(E)$ --which is a reductive Lie algebra pair with $\mathfrak{m}=E$ plus the holonomy of the connection, see e.g.~\cite[p.~206]{KN63}
for definitions. However, as it was pointed out to the authors by Yannick Voglaire, this assignment is well-known to be NOT functorial for the above categories, and we shall provide 
a simple explicit example of this defect in Section 2 of this note.

The first main result of this work is the explicit construction of a 
\textbf{left adjoint functor} $\mathbf{F}':\mathcal{LY}\to \mathcal{RLP}$ to the functor $\mathbf{G}'$ in the general Lie Yamaguti case, see Theorem \ref{TPrincipal}: this will define a universal construction,
a sort of \textit{free reductive Lie algebra pair $\mathfrak{g}(E)$ generated by the Lie-Yamaguti algebra $E$}. However, this adjunction does not define a categorical equivalence. Our result generalizes such a functorial
construction for the particular case of Lie triple systems which had already been sketched without details in 
\cite[p.~155-156]{Jac49} --around ten years before D.~Kan's definition of adjoint functors-- and finally duly formulated in \cite{Smi11}. \\
Next, in order
to reduce a little bit the above redundancy in the category $\mathcal{RLP}$
we have found it useful to define the subcategory $\mathrm{m}\mathcal{RLP}$ of all \textit{$\mathfrak{m}$-generated reductive Lie algebra pairs}, i.e.~where the
Lie algebra $\mathfrak{g}$ is generated by the subspace $\mathfrak{m}$.  It turns out that the inclusion functor
$\mathbf{J}:\mathrm{m}\mathcal{RLP}\to\mathcal{RLP}$ has a right adjoint
$\mathfrak{i}:\mathcal{RLP}\to \mathrm{m}\mathcal{RLP}$ assigning to each
$\big(\mathfrak{g},\mathfrak{h},\mathfrak{m}\big)$ the well-known ideal
$\mathfrak{i}(\mathfrak{g}):=\mathfrak{m}+[\mathfrak{m},\mathfrak{m}]\subset \mathfrak{g}$, see e.g.~\cite[p.~212, Thm.~52]{KN69} 
. Hence the subcategory
is \textit{coreflective}, see e.g.~\cite[p.~91]{Mac98} for definitions,
like the subcategory of all torsion abelian groups in the category of all abelian groups. It is not hard to see that
the above functors $\mathbf{F'}$ and $\mathbf{G'}$ factor over the subcategory, and we arrive at the following chain of adjunctions
\[
    \begin{tikzcd}
    	\mathcal{LY}
    	\arrow[r, yshift=0.7ex, "\mathbf{F}"]
    	          & m\mathcal{RLP}\arrow[l, yshift=-0.7ex, "\mathbf{G}"]
    	          \arrow[r, yshift=0.7ex, "\mathbf{J}"]
    	          &\mathcal{RLP}\arrow[l, yshift=-0.7ex, 
    	          "\boldsymbol{\mathfrak{i}}"]
    \end{tikzcd}
\]
where $\mathbf{F}'=\mathbf{JF}$ and $\mathbf{G}'=\mathbf{G}\boldsymbol{\mathfrak{i}}$.
The reductive Lie algebra pair $\mathbf{F}E=\mathfrak{g}(E)$ associated to a Lie-Yamaguti algebra $E$ by the functor $\mathbf{F}$ is an explicit quotient space of $E\oplus \Lambda^2E$.  It is curious and practical at the same time that the intermediate computations use the theory of Leibniz algebras, see e.g.~\cite{Lod92} for right Leibniz algebras, while here we use left Leibniz algebras (see e.g.~\cite{F}), which have also been used in \cite{KW01} for the enveloping Lie algebra. The difference between $\hat{\mathfrak{g}}(E)$ and $\mathfrak{g}(E)$ can be seen in the corollary that
an abelian Lie-Yamaguti algebra $E$ has an abelian enveloping Lie algebra $\hat{\mathfrak{g}}(E)=E$ whereas the functor $\mathbf{F}$ sends it to the non abelian, Heisenberg type
Lie algebra $\mathfrak{g}(E)=E\oplus \Lambda^2E$ with the obvious 
nonabelian Lie bracket.\\
The second main result is a possibility of situating the enveloping Lie algebra 
$\hat{\mathfrak{g}}(E)$ of a Lie-Yamaguti algebra based on $E$ in a categorical
manner: we show in Theorem \ref{TEnvLieAlgUniversal} that if we define subcategories $\mathcal{LY}\mathrm{s}$, $\mathrm{m}\mathcal{RLP}\mathrm{s}$, 
and $\mathrm{m}\mathcal{RLP}\mathrm{s}$ of  the above three categories by only allowing morphisms which are \textbf{surjective linear maps},
then the assignment of an enveloping Lie algebra to a Lie-Yamaguti algebra extends to surjective morphisms and thus defines a covariant functor $\hat{\mathfrak{g}}$
which turns out to be a \textbf{right adjoint functor} of
the restriction $\mathbf{G}_\mathrm{s}$ of the functor $\mathbf{G}$ to the subcategory $\mathrm{m}\mathcal{RLP}\mathrm{s}$ of $\mathrm{m}\mathcal{RLP}$. Note that these subcategories are respected by the two above functors $\mathbf{F}$
and $\mathbf{G}$. Hence
there is the following second adjunction which implies that every $\mathfrak{m}$-generated reductive Lie algebra pair $(\mathfrak{g},\mathfrak{h,\mathfrak{m}})$ is contained
in a chain of central extensions: 
\[
\begin{tikzcd}
	\mathrm{m}\mathcal{RLP}\mathrm{s}
	\arrow[r, yshift=0.7ex, "\mathbf{G}_\mathrm{s}"]
	&
	\mathcal{LY}\mathrm{s}
	\arrow[l, yshift=-0.7ex, "\hat{\mathfrak{g}}"]	
\end{tikzcd}
~~~\mathrm{implying}~~~
\begin{tikzcd}
	\mathfrak{g}(\mathfrak{m})
	\arrow[r, twoheadrightarrow, "\boldsymbol{\epsilon}_\mathfrak{g}"]
	&   \mathfrak{g}
	\arrow[r, twoheadrightarrow, "\boldsymbol{\eta}_{s\mathfrak{g}}"]
	&\hat{\mathfrak{g}}(\mathfrak{m})
\end{tikzcd}.
\]



\noindent \textbf{Acknowledgements}:
The authors would like to dedicate this work to the memory of Otto H. Kegel who
has passed away last year on his $91$st birthday. We all remember him as a gentle,
competent and very generous mathematician whose support has always been unlimited and 
valuable. In particular I, M.B., have learned true mathematics in innumerous conversations with him when I did my Diplomarbeit (Master thesis):
I profited a lot from his witty, but at the same time 
precise way of looking at mathematics which shaped my way of mathematical thinking and allowed me to pass from theoretical physics to mathematics. He also helped me in a substantial way to find a job at a university through his wide-spread network
of colleagues and friends. It is a great loss that he is non longer with us.
\\
The authors would also like to thank Yannick Voglaire for fruitful discussions,
and the anonymous referee for giving us very many very useful remarks which
improved the manuscript considerably.

\section{Reductive Lie algebra pairs and Lie-Yamaguti algebras}

 In this section $K$ is always a 
fixed
commutative associative unital ring containing the field of all 
rational numbers $\mathbb{Q}$ as a unital subring. All modules
are considered over $K$, and the symbol $\otimes$ is short for $\otimes_K$.
In view of MacLane's coherence theorem (see \cite[p.~165--170]{Mac98}
we can and will assume that the cartesian product $\times$ for sets and the tensor product $\otimes$ are associative. We shall write
$K$-multilinear maps in the old-fashioned non-tensorial
way with multiple arguments separated by commas.

Recall that a 
 \textit{reductive Lie algebra pair} $(\mathfrak{g},\mathfrak{h},\mathfrak{m})$,
 which in the following will be abbreviated by \textit{RLA pair}, is a triple consisting of a Lie algebra
$\mathfrak{g}$, a subalgebra $\mathfrak{h}\subset \mathfrak{g}$, and a complementary $K$-submodule $\mathfrak{m}\subset \mathfrak{g}$ of $\mathfrak{h}$
in $\mathfrak{g}$
such that $\mathfrak{m}$ is invariant under the adjoint action
of $\mathfrak{h}$, i.e.,
\begin{equation}\label{EqDefMStableAdH}
  [\mathfrak{h},\mathfrak{m}]\subset \mathfrak{m}
  ~~\Longleftrightarrow~~
  \forall~y\in\mathfrak{h},~\forall~z\in \mathfrak{m}:~
     [y,z]\in \mathfrak{m},
\end{equation}
see e.g.~\cite{NVH65} for cohomological obstructions to the existence of such complements. These pairs form the objects of a category
$\mathcal{RLP}$, where the set of all morphisms from
$(\mathfrak{g},\mathfrak{h},\mathfrak{m})$ to
$(\mathfrak{g}',\mathfrak{h}',\mathfrak{m}')$ consists
of all morphisms of Lie algebras $\phi:\mathfrak{g}\to
\mathfrak{g}'$ respecting the splitting, i.e.,
$\phi(\mathfrak{h})\subset \mathfrak{h}'$ and
$\phi(\mathfrak{m})\subset \mathfrak{m}'$. The restrictions of $\phi$ to $\mathfrak{h}$ and to $\mathfrak{m}$ are denoted by $\phi_\mathfrak{h}$ and
$\phi_\mathfrak{m}$, respectively. 
Recall that the particular case $[\mathfrak{m},\mathfrak{m}]\subset\mathfrak{h}$
leads to $\mathbb{Z}_2$-graded Lie algebras, also called
\textit{symmetric Lie algebras}, see e.g.~\cite[p.~225]{KN69}
or \cite{Hel78}.\\
Returning to the general case,
for a given RLA pair $(\mathfrak{g},\mathfrak{h},\mathfrak{m})$ denote the canonical projections $\mathfrak{g}\to \mathfrak{h}$
(with kernel $\mathfrak{m}$) and 
$\mathfrak{g}\to \mathfrak{m}$
(with kernel $\mathfrak{h}$) by
\begin{equation}
  \forall~x\in \mathfrak{g}:~~~x\mapsto x_\mathfrak{h}\in\mathfrak{h},~~~
   \mathrm{and}~~~
   x\mapsto x_\mathfrak{m}\in\mathfrak{m}
\end{equation}
where we adopt the notation from \cite[p.~191]{KN69}.
Obviously, we have $x=x_\mathfrak{h}+x_\mathfrak{m}$. Moreover, as an immediate consequence of (\ref{EqDefMStableAdH}), we have the following identities:
\begin{equation}\label{EqDefRLAadHpreservesHAndM}
 \forall~x\in \mathfrak{g},~\forall~y\in \mathfrak{h},~\forall~z\in \mathfrak{m}
   : ~~~[y,x]_\mathfrak{h}=
          \big[y,x_\mathfrak{h}\big],~~~ 
          \big[y,x\big]_\mathfrak{m}=
          \big[y,x_\mathfrak{m}\big],~~~\mathrm{and}~~~
          [z,x]_\mathfrak{h}=
          \big[z,x_\mathfrak{m}\big]_\mathfrak{h} .
   \end{equation}
   For any morphism $\phi: (\mathfrak{g},\mathfrak{h},\mathfrak{m})\to
   (\mathfrak{g}',\mathfrak{h}',\mathfrak{m}')$ we split the identity for the morphism into components:
   \begin{equation}\label{EqCompMMComponentsMorph}
   	\forall~z_1,z_2\in\mathfrak{m}:~~
   \phi_\mathfrak{h}	\left([z_1,z_2]_\mathfrak{h}\right)
   =\left[\phi_\mathfrak{m}(z_1),\phi_\mathfrak{m}(z_1)\right]_{\mathfrak{h}'}
   ~~\mathrm{and}~~
   \phi_\mathfrak{m}	\left([z_1,z_2]_\mathfrak{m}\right)
   =\left[\phi_\mathfrak{m}(z_1),\phi_\mathfrak{m}(z_1)\right]_{\mathfrak{m}'} .   
   \end{equation}

%
%


 \noindent Moreover,
call a subalgebra $\mathfrak{g}_1\subset
\mathfrak{g}$ --where $(\mathfrak{g},\mathfrak{h},\mathfrak{m})$ is an RLA pair--
\textit{transitive} iff $\mathfrak{m}\subset \mathfrak{g}_1$. Obviously, a transitive Lie subalgebra $\mathfrak{g}_1$ is an RLA pair $(\mathfrak{g}_1,\mathfrak{h}\cap \mathfrak{g}_1,
\mathfrak{m})$. As in \cite[p.212, Thm.52]{KN69}, define
\begin{equation}\label{EqDefLieIdealI}
	\mathfrak{i}(\mathfrak{g},\mathfrak{h},\mathfrak{m}) :=
  \mathfrak{i}(\mathfrak{g}):= 
     \mathfrak{m} + [\mathfrak{m},\mathfrak{m}],
\end{equation}
and we call the RLA pair $(\mathfrak{g},\mathfrak{h},\mathfrak{m})$  \textit{$\mathfrak{m}$-generated} iff $\mathfrak{g}=\mathfrak{i}(\mathfrak{g})$.
We have the following easy
\begin{proposition} With the above notation: $\mathfrak{i}(\mathfrak{g})$ is an
	ideal of $\mathfrak{g}$ and equal to the minimal transitive subalgebra of 
	$\mathfrak{g}$ w.r.t.~$\mathfrak{m}$. Moreover,
  the subclass $\mathrm{m}\mathcal{RLP}$ of all \textit{$\mathfrak{m}$-generated}
  RLA pairs is a full subcategory of $\mathcal{RLP}$, and there is the following
  adjunction of functors
  \begin{equation}
  	\begin{tikzcd}
  		 \mathrm{m}\mathcal{RLP}
  		\arrow[r, yshift=0.7ex, "\mathbf{J}"]
  		&\mathcal{RLP}\arrow[l, yshift=-0.7ex, 
  		"\boldsymbol{\mathfrak{i}}"]
  	\end{tikzcd}
  \end{equation}
  where $\mathbf{J}$ is the inclusion functor and $\boldsymbol{\mathfrak{i}}$
  is the functor assigning to $(\mathfrak{g},\mathfrak{h},\mathfrak{m})$
  the ideal $\mathfrak{i}(\mathfrak{g})$. The unit of the adjunction is an isomorphism and the counit is the injection
  $\mathfrak{i}(\mathfrak{g})\to \mathfrak{g}$ whence $m\mathcal{RLP}$ is thus a coreflective subcategory of $\mathcal{RLP}$, see \cite[p.~91]{Mac98}.
\end{proposition}
\begin{proof}
It is a routine check upon using $\mathfrak{h}$- and $\mathfrak{m}$-components
that $\mathfrak{i}(\mathfrak{g})$ is an ideal of $\mathfrak{g}$ which obviously is
transitive. Since every transitive subalgebra contains $\mathfrak{m}$ and 
$[\mathfrak{m},\mathfrak{m}]$, the ideal $\mathfrak{i}(\mathfrak{g})$ is contained in
every transitive subalgebra, and thus the first statement is clear. Next,
the fact that
\begin{equation}\label{EqCompHCompOfLieIdealI}
	\mathfrak{h}\cap \mathfrak{i}(\mathfrak{g})
	=K\mathrm{span}\big\{[z_1,z_2]_\mathfrak{h}
	~\big|~z_1,z_2\in\mathfrak{m}\big\}
\end{equation}
shows that $\big(\mathfrak{i}(\mathfrak{g}),\mathfrak{i}(\mathfrak{g})\cap \mathfrak{h},\mathfrak{m}\big)$ is in $\mathrm{m}\mathcal{RLP}$ and
implies that each morphism $\mathfrak{g}\to\mathfrak{g}'$ in $\mathcal{RLP}$ restricts to a morphism of RLA pairs $\mathfrak{i}(\mathfrak{g})\to
\mathfrak{i}(\mathfrak{g}')$, see (\ref{EqCompMMComponentsMorph}), whence
$\boldsymbol{\mathfrak{i}}$ is a well-defined functor. The adjunction properties are easily checked upon observing that for any
$\mathfrak{m}$-generated RLA pair $\mathfrak{g}$ and each morphism
$\mathfrak{g}\to\mathfrak{g}'$ of RLA pairs automatically corestricts to
$\mathfrak{i}(\mathfrak{g}')\subset \mathfrak{g}'$. 
\end{proof}

\noindent 
Returning to an arbitrary RLA pair $(\mathfrak{g},\mathfrak{h},\mathfrak{m})$
one defines the following $K$-bilinear map
$T:\mathfrak{m}\times\mathfrak{m}\to \mathfrak{m}$
and $K$-trilinear map
$R:(\mathfrak{m}\times\mathfrak{m})\times \mathfrak{m}\to \mathfrak{m}$ by 
\begin{eqnarray}
\forall~z,z'\in
\mathfrak{m}:~~~	T(z,z') & := &  -[z,z']_\mathfrak{m},
	   \label{EqDefRLAMapT}\\
\forall~z,z',z''\in
\mathfrak{m}:~~~	R(z,z',z'')=:R(z,z')z'' 
	  & := & -\big[[z,z']_\mathfrak{h},z''\big],
	  \label{EqDefRLAMapR}
\end{eqnarray}
compare \cite[p.~193,~Thm.~2.6,(1),(2)]{KN69}: this explains the unusual notation for the pair of brackets $(T,R)$
because of the relation to the \textit{torsion tensor} $T$ and the \textit{curvature tensor} $R$ of a particular canonical connection in the differential geometry of reductive homogeneous spaces, see e.g.~\cite[p.~190]{KN69}. 

\noindent Using the following notation for \textit{cyclic sums} borrowed from \cite[p.~135]{KN63}: given any $K$-modules
$V,W$ and any trilinear map $\Xi:V\times V\times V\to W$ we write for all
$a,b,c\in V$
\begin{equation}
	\forall~a,b,c\in V:~~~\mathfrak{S}_{(a,b,c)}\big(\Xi(a,b,c)\big):=
	\Xi(a,b,c)+\Xi(b,c,a)+\Xi(c,a,b).
\end{equation}
We obtain the 
following set of identities for $T$ and $R$ which can easily be deduced
from the definitions (\ref{EqDefRLAMapT}) and (\ref{EqDefRLAMapR}) by taking certain $\mathfrak{h}$- and $\mathfrak{m}$-components in the Jacobi identity for
the Lie bracket of $\mathfrak{g}$ and from  (\ref{EqCompMMComponentsMorph}):
\begin{proposition}\label{PLRAtoLieY}
	Let $(\mathfrak{g},\mathfrak{h},\mathfrak{m})$ be an arbitrary RLA pair. Then
	for all $z_1,z_2, z_3,z_4,z\in\mathfrak{m}$
	 the maps $T$ and $R$ satisfy the following six identities:
	\begin{eqnarray}
	T(z_1,z_2) & = & -T(z_2,z_1), \label{EqDefRLAToLieYAntisymT}\\
	R(z_1,z_2)z & = & -R(z_2,z_1)z, \label{EqDefRLAToLieYAntisymR}\\
	\mathfrak{S}_{(z_1,z_2,z_3)} 
	\left(R(z_1,z_2)z_3-T\big(T(z_1,z_2),z_3\big)\right)
	& = & 0, \label{EqDefRLAToLieYCyclicRandTTBianchi1}\\
	\mathfrak{S}_{(z_1,z_2,z_3)} 
	\left(R\big(T(z_1,z_2),z_3\big)z\right) & = &
	 0,\label{EqDefRLAToLieYCyclicRT}\\
	R(z_1,z_2)\big(T(z_3,z_4)\big)& = &
	  T\big(R(z_1,z_2)z_3,z_4\big)+T\big(z_3,R(z_1,z_2)z_4\big),
	   \label{EqDefRLAToLieYDeriROnT}\\
	R(z_1,z_2)\big(R(z_3,z_4)z\big) & = &
	R\big(R(z_1,z_2)z_3,z_4\big)z
	+R\big(z_3,R(z_1,z_2)z_4\big)z\nonumber \\
& &	+R(z_3,z_4)\big(R(z_1,z_2)z\big).
\label{EqDefRLAToLieYDeriROnR}
\end{eqnarray}
Moreover, let $\phi:(\mathfrak{g},\mathfrak{h},\mathfrak{m})\to
(\mathfrak{g}',\mathfrak{h}',\mathfrak{m}')$ be a morphism
of RLA pairs. Then its $\mathfrak{m}$-component
$\phi_\mathfrak{m}:\mathfrak{m}\to \mathfrak{m}'$ maps 
 $T$ and $R$ to their corresponding maps $T'$ and $R'$, respectively, i.e., 
\begin{eqnarray}
     \forall~z_1,z_2\in \mathfrak{m}:~~~\phi_\mathfrak{m}\big(T(z_1,z_2)\big)
     & = & T'\big(\phi_\mathfrak{m}(z_1),
        \phi_\mathfrak{m}(z_2)\big), \label{EqCompPhiSubMIntertwTTPrime}\\
     \forall~z, z_1,z_2\in \mathfrak{m}:~~~\phi_\mathfrak{m}\big(R(z_1,z_2)z\big)
     & = & R'\big(\phi_\mathfrak{m}(z_1),
     \phi_\mathfrak{m}(z_2)\big)(\phi_\mathfrak{m}(z)).
     \label{EqCompPhiSubMIntertwRRPrime}
\end{eqnarray}
\end{proposition}
\noindent Note that the identities (\ref{EqDefRLAToLieYAntisymT}) --
(\ref{EqDefRLAToLieYDeriROnR}) are a particular case of the Bianchi identities and some consequences for covariantly constant torsion and curvature, see
e.g.~\cite[p.~135, Thm.~5.3]{KN63}

\noindent 
The preceding results give rise to the following definition which is due to Kiyosi Yamaguti 
\cite[p.~157, Def.~2.1]{Yam58}:
\begin{definition}
	Let $E$ be a $K$-module equipped with a $K$-bilinear map $T:E\times E\to E$ and a $K$-trilinear map
	$R:(E\times E)\times E\to E$ which satisfy the six identities (\ref{EqDefRLAToLieYAntisymT}) --
	(\ref{EqDefRLAToLieYDeriROnR}). Then the pair
	$(E,T,R)$ is called a \textbf{Lie-Yamaguti algebra}
	(short: LY algebra).\\
	Let $(E',T',R')$ be another Lie-Yamaguti algebra.
	A $K$-linear map $\psi:E\to E'$  is called
	a morphism of Lie-Yamaguti algebras if for all
	$z_1,z_2,z_3\in E$
	\begin{equation}\label{EqDefMorphLY} 
	 \psi\big(T(z_1,z_2)\big)
	 =  T'\big(\psi(z_1),
	\psi(z_2)\big)~~\mathrm{and}~~
	\psi\big(R(z_1,z_2)z_3\big)
	 =  R'\big(\psi(z_1),
	\psi(z_2)\big)(\psi(z_3)).
	\end{equation}
	Hence $\psi$ intertwines the corresponding maps $T,R,T',R'$ as $\phi_\mathfrak{m}$ does in 
	(\ref{EqCompPhiSubMIntertwTTPrime})
	and (\ref{EqCompPhiSubMIntertwRRPrime}).
\end{definition}
\noindent K.~Yamaguti has denoted the map $(v,w,z)\mapsto -R(v,w)z$ by a triple bracket
$(v,w,z)\mapsto[v,w,z]$ and the map $(v,w)\mapsto -T(v,w)$ by a multiplication
$(v,w)\mapsto v\circ w$, and he has called these objects `Lie triple algebras'. The motivating  particular case $T=0$ reduces the six conditions of Proposition \ref{PLRAtoLieY} to
(\ref{EqDefRLAToLieYAntisymR}), (\ref{EqDefRLAToLieYCyclicRandTTBianchi1}), and
(\ref{EqDefRLAToLieYDeriROnR}) in which case the Lie-Yamaguti algebra is called a \textit{Lie triple system}, see e.g.~\cite{Jac49}. The other extreme case $R=0$
reduces to (\ref{EqDefRLAToLieYAntisymT}) and (\ref{EqDefRLAToLieYCyclicRandTTBianchi1}) which means that $(E,T)$ is a Lie algebra.\\
It follows at once that Lie-Yamaguti algebras together with their
morphisms form a category $\mathcal{LY}$, and that there
is an obvious functor $\mathbf{G}': \mathcal{LY} \longleftarrow  \mathcal{RLP}$
assigning to each RLA pair $(\mathfrak{g},\mathfrak{h},\mathfrak{m})$ the LY algebra
$(\mathfrak{m},T,R)$ according to
(\ref{EqDefRLAMapT}) and (\ref{EqDefRLAMapR}), and to
each RLA pair morphism $\phi$ the component $\phi_\mathfrak{m}$. It is easy to see
that $\mathbf{G}'$ factorizes as $\mathbf{G}'=\mathbf{G}\boldsymbol{\mathfrak{i}}$:
\begin{equation}\label{EqDefFunctorG}
    \begin{tikzcd}
    	\mathcal{LY} &
    	    \mathrm{m}\mathcal{LRP}
    	      \arrow[l, "\mathbf{G}"'] &
    	        \mathcal{LRP}
    	        \arrow[l, "\boldsymbol{\mathfrak{i}}"']
    \end{tikzcd}
\end{equation}
where $\mathbf{G}$ is the `restriction' of $\mathbf{G}'$ to $\mathrm{m}\mathcal{LRP}$.

\section{Two adjoint functors from Lie-Yamaguti algebras to
	  reductive Lie algebra pairs}

There is a well-established assignment of a RLA pair  
$\big(\hat{\mathfrak{g}}(E), \hat{\mathfrak{h}}(E),E\big)$ to every Lie-Yamaguti algebra $(E,T,R)$ --which is called its
\textit{enveloping Lie algebra}--
in the following way, see e.g.~\cite[p.~158, Prop.~2.1]{Yam58}: consider the $K$-submodule $\hat{\mathfrak{h}}(E)$ of
$\mathrm{Hom}_K(E,E)$ which is spanned by
all the $K$-linear maps of the form 
$\hat{R}(v_1,v_2):z\mapsto R(v_1,v_2)z$, i.e.,
\begin{equation}\label{EqDefHatHOfE}
\hat{\mathfrak{h}}(E):=K\mathrm{span}	
       \big\{ -\hat{R}(v_1,v_2)\in\mathrm{Hom}_K(E,E)~\big|~
                   v_1,v_2\in E\big\}.
\end{equation}
Thanks to (\ref{EqDefRLAToLieYDeriROnR}) the $K$-submodule $\hat{\mathfrak{h}}(E)$, equipped with the commutator of linear maps, is a Lie subalgebra
of $\mathrm{Hom}_K(E,E)$. Define 
\begin{equation}\label{EqDefEnvelLieAlg}
   \hat{\mathfrak{g}}(E):= \hat{\mathfrak{h}}(E)\oplus E
\end{equation}
with the bracket 
\begin{equation}\label{EqDefEnvelLieAlgBracket}
	\forall~\xi,\eta\in \hat{\mathfrak{h}}(E)~\forall~v,w\in E:
	[\xi+v,\eta+w]^\wedge:=\big(\xi\circ \eta-\eta\circ\xi-\hat{R}(v,w)\big)
	                ~~+~~\big(\xi(w)-\eta(v)-T(v,w)\big).
\end{equation}
Since the 1950's it is well known and not hard to check by using the Lie-Yamaguti identities that $\big(\hat{\mathfrak{g}}(E),\hat{\mathfrak{h}}(E),E\big)$ is a reductive Lie algebra pair.\\
Although the enveloping Lie algebra has turned out to be very useful in differential geometry, the algebraic draw-back is the fact that 
the assignment of
a Lie-Yamaguti algebra to its enveloping Lie algebra cannot be defined on
morphisms of Lie-Yamaguti algebras in a functorial way: it follows that this assignment
does \textbf{not} lead to a functor $\mathcal{LY}\to\mathcal{RLP}$
\footnote{We owe this remark to a conversation with
Yannick Voglaire} as the following counter-example shows:

\begin{example}\label{ExCounterexample}
We shall construct two Lie triple systems, hence particular LY algebras with
vanishing binary bracket, by means of the following well known construction:
Recall that for an arbitrary Lie algebra $\big(\mathfrak{l},[~,~]\big)$
there is always a Lie triple system $\big(E(\mathfrak{l})=\mathfrak{l},R_\mathfrak{l}\big)$  where the ternary bracket $R_\mathfrak{l}$ is defined by
\[
\forall~x,y,z\in\mathfrak{l}:~~~R_\mathfrak{l}(x,y)z:=-\big[[x,y],z\big].
\]
Furthermore, it is easy to check that the assignment of $\big(\mathfrak{l},[~,~]\big)$ to $\big(E(\mathfrak{l}),R_\mathfrak{l}\big)$
defines an obvious functor from the category of all Lie algebras to the category of all Lie triple systems. For each $x\in\mathfrak{l}$ let $\mathrm{ad}_x:\mathfrak{l}\to \mathfrak{l}$ denote
the usual adjoint representation $z\mapsto \mathrm{ad}_x(z)=[x,z]$.
It is not hard to see from the above definition of $R_\mathfrak{l}$ that the enveloping Lie algebra $\hat{\mathfrak{g}}(E(\mathfrak{l}))$ of $\big(E(\mathfrak{l}),R_\mathfrak{l}\big)$
is given by the RLA pair $\big(\mathrm{ad}_{[\mathfrak{l},\mathfrak{l}]}\oplus \mathfrak{l},\mathrm{ad}_{[\mathfrak{l},\mathfrak{l}]},\mathfrak{l}\big)$
where $\mathrm{ad}_{[\mathfrak{l},\mathfrak{l}]}\subset \mathrm{Hom}_\mathbb{K}(\mathfrak{l},\mathfrak{l})$ denotes the Lie subalgebra
of the Lie algebra $\mathrm{Hom}_\mathbb{K}(\mathfrak{l},\mathfrak{l})$ spanned by all linear maps of the form $\mathrm{ad}_{[x,y]}$ for some $x,y\in\mathfrak{l}$.
According to (\ref{EqDefEnvelLieAlgBracket}) the Lie bracket on $\hat{\mathfrak{g}}(E(\mathfrak{l}))$ is given by
\[
	\forall~\xi,\eta\in \mathrm{ad}_{[\mathfrak{l},\mathfrak{l}]}~
	\forall~x,y\in\mathfrak{l}:~~
	\big[\xi+x,\eta+y\big]^\wedge= 
	\big(\xi\circ \eta-\eta\circ\xi+\mathrm{ad}_{[x,y]}\big)
	~+~\big(\xi(y)-\eta(x)\big).
\]	
Consider a field
$\mathbb{K}$ of characteristic $0$, let $\mathfrak{l}$ be the abelian Lie algebra $\mathbb{K}^2$, and let $\mathfrak{l}'$ be the six-dimensional nilpotent Lie algebra of
all strictly upper triangular $4\times 4$ matrices.  Let $e_1,e_2$ be the canonical
basis of $\mathbb{K}^2$. Denoting by $E_{ij}$ the standard elementary 
$4\times 4$-matrices for all $1\leq i,j\leq 4$ it is immediate that the six matrices $E_{12},E_{13}, E_{14}, E_{23}, E_{24}, E_{34}$ form a basis for the
Lie algebra $\mathfrak{l}'$. Recall the equations for the Lie brackets
\[
   \forall~1\leq i,j,k,l\leq 4:~~~  [E_{ij},E_{kl}] =\delta_{jk}E_{il}-\delta_{il}E_{kj}
\]
with the usual Kronecker delta $\delta_{ij}=1$ if $i=j$ and $\delta_{ij}=0$ if
$i\not = j$. We observe that the $4$-dimensional subspace $I\subset \mathfrak{l}'$
spanned by $E_{13},E_{14},E_{24}, E_{34}$ forms an ideal of $\mathfrak{l}'$, and that the linear map $p:\mathfrak{l}'\to\mathfrak{l}$ sending $E_{12}$ to $e_1$
and 
$E_{23}$ to $e_2$ and $I$ to $\{0\}$ is a surjective morphism of Lie algebras.
Hence, we get a morphism --also denoted by $p$-- of the Lie triple system
$\big(E(\mathfrak{l}'),R_{\mathfrak{l}'}\big)$ to the abelian Lie triple system $\big(E(\mathfrak{l}),0\big)$. On the other hand, consider the linear map
$i:\mathfrak{l}\to \mathfrak{l}'$ which sends $e_1$ to $E_{12}$ and $e_2$ to
$E_{23}$. This is \textbf{not} a morphism of Lie algebras because
$[e_1,e_2]=0$, but $[E_{12},E_{23}]=E_{13}\not = 0$. However, the map $i$ is a morphism
of the associated Lie triple systems since
\[
R_{\mathfrak{l}'}\big(i(e_1),i(e_2)\big)(i(e_1))=
  R_{\mathfrak{l}'}(E_{12},E_{23})E_{12}=-\big[[E_{12},E_{23}],E_{12}\big]
     =-[E_{13},E_{12}]=0=i\big(R_\mathfrak{l}(e_1,e_2)e_1\big)
\]
and likewise $R_{\mathfrak{l}'}(E_{12},E_{23})E_{23}=0$. It follows that we get
the following diagram of Lie triple systems
\[
\begin{tikzcd}
	E(\mathfrak{l})
	   \arrow[r, yshift=0.7ex, "i"] 
	     & E(\mathfrak{l}')
	     \arrow[l, yshift=-0.7ex, "p"] 
	     ~~~\mathrm{with}~~~ p\circ i=\mathrm{id}_{E(\mathfrak{l})}.
\end{tikzcd}
\]
Since $\mathfrak{l}$ is abelian, we have $\mathrm{ad}_{\mathfrak{l}}=\{0\}$.
Hence, the enveloping Lie algebra $\hat{\mathfrak{g}}(E(\mathfrak{l}))$
is isomorphic to the RLA pair $\big(\mathfrak{l},0,\mathfrak{l}$, hence, to the abelian Lie algebra $\mathfrak{l}$. Moreover, the enveloping Lie algebra $\hat{\mathfrak{g}}(E(\mathfrak{l}'))$ is easily computed to be isomorphic
to the RLA pair
\[
   \hat{\mathfrak{g}}(E(\mathfrak{l}')) 
   \cong \left(\mathbb{K}\mathrm{span}\left\{\mathrm{ad}_{E_{13}},
        \mathrm{ad}_{E_{24}}\right\}
     \oplus \mathfrak{l}', \mathbb{K}\mathrm{span}\left\{\mathrm{ad}_{E_{13}},
     \mathrm{ad}_{E_{24}}\right\}, \mathfrak{l}'\right).
\]
Note that both linear maps $\mathrm{ad}_{E_{13}}$ and $\mathrm{ad}_{E_{24}}$
vanish on $E_{23}, E_{13}, E_{24}$, and $E_{14}$, that $\mathrm{ad}_{E_{13}}$ vanishes on $E_{12}$ and sends $E_{34}$ to the central element $E_{14}$,
whereas $\mathrm{ad}_{E_{24}}$ vanishes on $E_{34}$ and sends $E_{12}$ to the central element $-E_{14}$. It follows that the linear maps $\mathrm{ad}_{E_{13}}$ and $\mathrm{ad}_{E_{24}}$ are linearly independent.\\
Suppose there was a functor $\mathbf{H}:\mathcal{LY}\to \mathcal{RLP}$ assigning to each Lie-Yamaguti algebra $E$ its enveloping Lie algebra 
$\mathbf{H}E:=\hat{\mathfrak{g}}(E)$. Applying the functor $\mathbf{H}$ to the above diagram we would get
a corresponding diagram
\[
\begin{tikzcd}
	\hat{\mathfrak{g}}(E(\mathfrak{l}))
	\arrow[r, yshift=0.7ex, "\mathbf{H}i"] 
	& \hat{\mathfrak{g}}(E(\mathfrak{l}'))
	\arrow[l, yshift=-0.7ex, "\mathbf{H}p"] 
	~~~\mathrm{with}~~~ \mathbf{H}p\circ \mathbf{H}i
	=\mathbf{H}\mathrm{id}_{E(\mathfrak{l})}=
	\mathrm{id}_{\hat{\mathfrak{g}}(E(\mathfrak{l}))}.
\end{tikzcd}
\]
It follows that the morphism of RLA pairs $\mathbf{H}i$ would be injective whereas the morphism of RLA pairs $\mathbf{H}p$ would be surjective: this implies that
the image $\mathbf{H}i\big(\hat{\mathfrak{g}}(E(\mathfrak{l}))\big)$ would be a
two-dimensional abelian subalgebra of the $\mathfrak{l}'$-part of
$\hat{\mathfrak{g}}(E(\mathfrak{l}'))$ which would be a vector space complement 
to the subspace $I\subset \mathfrak{l}'$. A basis of 
$\mathbf{H}i\big(\hat{\mathfrak{g}}(E(\mathfrak{l}))\big)$
could then be given by the matrices $E_{12}+A$ and $E_{23}+B$ with
some $A,B\in I$. An easy computation shows that there would be $\lambda\in\mathbb{K}$ such that
\[
   \big[E_{12}+A,E_{23}+B\big]^\wedge=
   \mathrm{ad}_{[E_{12}+A,E_{23}+B]}=\mathrm{ad}_{E_{13}}+\lambda \mathrm{ad}_{E_{24}}\not =0
\]
thanks to the linear independence of $\mathrm{ad}_{E_{13}}$ and
$\mathrm{ad}_{E_{24}}$. But this is in contradiction to the fact that
the subspace $\mathbf{H}i\big(\hat{\mathfrak{g}}(E(\mathfrak{l}))\big)$
is abelian. Therefore such a functor $\mathbf{H}$ cannot exist.
\hfill $\Box$ 
\end{example}

\noindent In this section we should like to construct a left adjoint
functor 
\begin{equation}
	\mathbf{F}:\mathcal{LY} \longrightarrow  \mathrm{m}\mathcal{RLP}
\end{equation}
to the functor $\mathbf{G}$, see (\ref{EqDefFunctorG}) which will give a sort of `free object in $\mathrm{m}\mathcal{RLP}$ generated by a given Lie-Yamaguti algebra'.
A similar construction has already been done for the more particular Lie triple systems, see
\cite{Smi11}.\\
In order to do so, we start with a Lie-Yamaguti algebra
$(E,T,R)$, and we form the $K$-modules
\begin{equation}\label{EqDefTildeGE}
 \tilde{\mathfrak{m}}(E) := E,~~~~
 \tilde{\mathfrak{h}}(E) :=\Lambda^2E,~~~~\tilde{\mathfrak{g}}(E):= E\oplus \Lambda^2E=
 \tilde{\mathfrak{m}}_E\oplus \tilde{\mathfrak{h}}(E).
\end{equation}
We define the following $K$-bilinear multiplication 
$[~,~]^\sim:\tilde{\mathfrak{g}}(E)\times \tilde{\mathfrak{g}}(E)\to \tilde{\mathfrak{g}}(E)$ for all
 $v,v_1,v_2$, $w,w_1,w_2\in E$
\begin{eqnarray}
 [v_1,v_2]^\sim & := & -T(v_1,v_2)+ v_1\wedge v_2, 
                    \label{EqDefTildeBracketEE} \\
 ~[w_1\wedge w_2, v]^\sim & := &
      -R(w_1,w_2)v =: -[v,w_1\wedge w_2]^\sim,
               \label{EqDefTildeBracketEW2E} \\
 ~[v_1\wedge v_2,w_1\wedge w_2]^\sim & := &
   -\big(R(v_1,v_2)w_1\big)\wedge w_2
   -w_1\wedge \big(R(v_1,v_2)w_2\big).
      \label{EqDefTildeBracketW2EW2E}
\end{eqnarray}
We shall see in Lemma \ref{LFELeibnizProperties} that
this bracket is well-defined whence the pair
$\big(\tilde{\mathfrak{g}},[~,~]^\sim\big)$
will be a non-associative algebra. However,
it will turn out that the above bracket 
$\big[~,~\big]^\sim$ is not a Lie bracket,
it is a priori not even antisymmetric, see 
(\ref{EqDefTildeBracketW2EW2E}).
In order to get an idea for the modification of this bracket $[~,~]^\sim$ on a factor module, it has turned out to be convenient
to use the theory of the slightly more general (left) \textbf{Leibniz algebras}, see e.g.~\cite[p.~332]{Lod92} for right Leibniz algebras and e.g.~\cite{F} for left Leibniz algebras:  for a given
$K$-module $V$ and $K$-bilinear map $[~,~]^\vee:V\times V\to V$
we set:
\begin{equation}\label{EqDefLeibnizIdentityMapL}
	\forall~\xi_1,\xi_2,\xi_3\in V:~~~
 \mathbf{L}(\xi_1,\xi_2,\xi_3):=
    \big[\xi_1,[\xi_2,\xi_3]^\vee\big]^\vee
     -\big[[\xi_1,\xi_2]^\vee,\xi_3\big]^\vee
     -\big[\xi_2,[\xi_1,\xi_3]^\vee\big]^\vee.
\end{equation}
The pair $(V,[~,~]^\vee)$ is called a Leibniz algebra iff
$\mathbf{L}(\xi_1,\xi_2,\xi_3)=0$ for all
$\xi_1,\xi_2,\xi_3\in V$. Recall that every Leibniz algebra whose bracket is antisymmetric is a Lie algebra. Moreover, recall that for every Leibniz algebra $\big(V,[~,~]^\vee\big)$ the $K$-submodule
\begin{equation}\label{EqDefLeibnizQuadraticIdeal}
  q(V)
   :=K\mathrm{span}\big\{[v,w]^\vee+[w,v]^\vee~|~
   v,w\in V\big\}
\end{equation}
is a two-sided abelian ideal of $\big(V,[~,~]^\vee\big)$--
meaning that the induced multiplication vanishes--,
and that the quotient Leibniz algebra 
$\overline{V}:=V/q(V)$ is a Lie algebra.
We first need the following
\begin{lemma}\label{LFELeibnizProperties}
	Let $(E,T,R)$ be a Lie-Yamaguti algebra. Then the
	bilinear bracket $[~,~]^\sim$ on $\tilde{\mathfrak{g}}(E)$ given by eqs (\ref{EqDefTildeBracketEE}),
	(\ref{EqDefTildeBracketEW2E}), and
	(\ref{EqDefTildeBracketW2EW2E}) is well-defined and 
	satisfies the following six equations for all
	$v,v_1,v_2,w,w_1,w_2,z,z_1,z_2\in E$
	\begin{eqnarray}
	 \mathbf{L}(v_1,v_2,z)& = &
	     \mathfrak{S}_{(v_1,v_2,z)}
	        \left(T(v_1,v_2)\wedge z\right)~\in~
	           \tilde{\mathfrak{h}}(E),
	           \label{EqCompLeibnizMMM} \\
	 \mathbf{L}(v,w_1\wedge w_2, z)
	 &= & -~\mathbf{L}(w_1\wedge w_2,v,z)~=~0,
	 \label{EqCompLeibnizMHM}\\
	 \mathbf{L}(v_1\wedge v_2,w_1\wedge w_2, z)
	 & = & 0,
	 \label{EqCompLeibnizHHM}\\
	 \mathbf{L}(v_1, v_2,w_1\wedge w_2) & = &
	    \big(R(v_1,v_2)w_1\big)\wedge w_2
	    +w_1\wedge \big(R(v_1,v_2)w_2\big)\nonumber \\
	& &    +\big(R(w_1,w_2)v_1\big)\wedge v_2
	    +v_1\wedge \big(R(w_1,w_2)v_2\big)\nonumber \\
	& = &  -[v_1\wedge v_2,w_1\wedge w_2]^\sim
	        -[w_1\wedge w_2,v_1\wedge v_2]^\sim
	        ~\in \tilde{\mathfrak{h}}(E),
	      \label{EqCompLeibnizMMH} \\
	 \mathbf{L}(z,v_1\wedge v_2,w_1\wedge w_2) & = & 
	 -\mathbf{L}(v_1\wedge v_2,z,w_1\wedge w_2)~=~0,
	   \label{EqCompLeibnizMHH}\\
	 \mathbf{L}(v_1\wedge v_2,w_1\wedge w_2,z_1\wedge z_2) 
	     & = & 0.
	    \label{EqCompLeibnizHHH} 
	\end{eqnarray}
\end{lemma}
\begin{proof} The bracket $[~,~]^\sim$ is well-defined thanks to the antisymmetry of $T$ and $R$ in the left two
	arguments and to the fact that the exterior algebra
	$\Lambda E$ is the free graded commutative algebra generated by $E$. The six identities are straight-forward computations:
	 (\ref{EqCompLeibnizMMM}) follows from identity (\ref{EqDefRLAToLieYCyclicRandTTBianchi1}),  (\ref{EqCompLeibnizMHM})
	follows from identity (\ref{EqDefRLAToLieYDeriROnT}),  (\ref{EqCompLeibnizHHM}) follows from identity (\ref{EqDefRLAToLieYDeriROnR}),
	 (\ref{EqCompLeibnizMMH}) follows from identity (\ref{EqDefRLAToLieYDeriROnT}), ( \ref{EqCompLeibnizMHH}) follows from identity (\ref{EqDefRLAToLieYDeriROnR}), and  (\ref{EqCompLeibnizHHH})
	is deduced from identity (\ref{EqDefRLAToLieYDeriROnR}).
\end{proof}
\noindent Now the definition of the bracket
$[~,~]^\sim$ in eqs (\ref{EqDefTildeBracketEE}),
(\ref{EqDefTildeBracketEW2E}) and (\ref{EqDefTildeBracketW2EW2E}), and statement
(\ref{EqCompLeibnizHHH}) of Lemma \ref{LFELeibnizProperties}
implies the obvious
\begin{corollary}\label{CHTildeLeibniz}
	With the hypotheses of the preceding Lemma
	\ref{LFELeibnizProperties} we have
	\begin{enumerate}
		\item $\tilde{\mathfrak{h}}(E)$ is a subalgebra
		  of $\tilde{\mathfrak{g}}(E)$ and satisfies the
		  Leibniz identity.
		 \item  We have
		   \begin{equation}
		    \big[\tilde{\mathfrak{h}}(E),
		    \tilde{\mathfrak{m}}_E\big]^\sim
		    \subset \tilde{\mathfrak{m}}_E
		    ~~~\mathrm{and}~~~
		    \big[\tilde{\mathfrak{m}}_E,
		    \tilde{\mathfrak{h}}(E)\big]^\sim
		    \subset \tilde{\mathfrak{m}}_E
		   \end{equation}
		   and identities (\ref{EqCompLeibnizHHM}) and (\ref{EqCompLeibnizMHH})
		   imply that $\tilde{\mathfrak{m}}_E=E$ is a so-called symmetric Leibniz bimodule of $\tilde{\mathfrak{h}}(E)$, see Section 3 in \cite{F}.
	\end{enumerate}
\end{corollary}
\noindent Next, it is reasonable to define the following $K$-submodules of $\tilde{\mathfrak{h}}(E)\subset \tilde{\mathfrak{g}}(E)$ which describe the failure of the Leibniz identity of
$\big(\tilde{\mathfrak{g}}(E),[~,~]^\sim\big)$:
\begin{eqnarray}
   \tilde{I}_1 & := &
      K\mathrm{Span}\Big\{
        \big[v_1\wedge v_2, w_1\wedge w_2\big]^\sim
        +\big[w_1\wedge w_2, v_1\wedge v_2\big]^\sim~
          \Big|~v_1,v_2,w_1,w_2\in E
      \Big\}, 
      \label{EqDefIdealOne}\\
   \tilde{I}_2 & := &
   K\mathrm{Span}\left\{
   \mathfrak{S}_{(v,v_1,v_2)}\big(v\wedge T(v_1,v_2)\big)~
   \big|~v,v_1,v_2\in E
   \right\}, \label{EqDefIdealTwo}\\ 
   \tilde{I}(E) & := &   \tilde{I}_1 +\tilde{I}_2.
          \label{EqDefIdealOnePlusIdealTwo}
\end{eqnarray}
\begin{lemma}\label{LFunctorFComputations}
  	With the notation of the preceding Lemma \ref{LFELeibnizProperties}: 
  	\begin{enumerate}
  		\item $\tilde{I}_1$ and
  	$\tilde{I}_2$, and hence, $\tilde{I}(E)$ are two-sided ideals of
  	$\tilde{\mathfrak{g}}$ contained in the subalgebra $\tilde{\mathfrak{h}}$.
  	\item The factor algebra 
  	\begin{equation}\label{EqDefGE}
  		\mathfrak{g}(E):= 
  		 \tilde{\mathfrak{g}}(E)/\tilde{I}(E)
  	\end{equation}
  	gives rise to a $\mathfrak{m}$-generated reductive Lie algebra pair
  	$\big(\mathfrak{g}(E),\mathfrak{h}(E),E\big)$
  	where
  	\begin{equation}\label{EqDefHE}
  	   \mathfrak{h}(E):= 
  	   \tilde{\mathfrak{h}}(E)/\tilde{I}(E).
  	\end{equation}
  	\item Let $(\mathfrak{g}',\mathfrak{h}',\mathfrak{m}')$
  	be a RLA pair with Lie bracket $[~,~]'$. Let
  	$\chi:(E,T,R)\to (\mathfrak{m}',T',R')$ be a morphism of
  	Lie-Yamaguti algebras, where $T'$ and $R'$ are defined as in (\ref{EqDefRLAMapT}) and (\ref{EqDefRLAMapR}). Then the $K$-linear map $\tilde{\chi}:\tilde{\mathfrak{g}}(E)\to \mathfrak{g}'$ given by 
  	\begin{equation}\label{EqDefFofPsi}
  		\forall~v,v_1,v_2\in E:~~~~~
  	  \tilde{\chi}(v):=\chi(v)~~~\mathrm{and}~~~
  	  \tilde{\chi}\big(v_1\wedge v_2
  	  \big):=
  	   \big[\chi(v_1),\chi(v_2)\big]'_{\mathfrak{h}'}
  	\end{equation}
  	is a morphism of nonassociative algebras and passes to the quotient to define a morphism of RLA pairs $\check{\chi}:\mathfrak{g}(E)\to \mathfrak{g}'$.
  \end{enumerate}
\end{lemma}
\begin{proof} \textit{i.)} According to  Corollary
	\ref{CHTildeLeibniz}, the submodule $\Lambda^2E=\tilde{\mathfrak{h}}(E)$
	is a subalgebra of
	 $\big(\tilde{\mathfrak{g}}(E),[~,~]^\sim\big)$ and satisfies the Leibniz identity. Since the $K$-submodule $I_1$ obviously coincides with the
	 ideal of squares $q\big(\tilde{\mathfrak{h}}(E)\big)$,
	 see eqn (\ref{EqDefLeibnizQuadraticIdeal}), it is an
	 abelian two-sided ideal of the subalgebra
	 $\tilde{\mathfrak{h}}(E)$, i.e.~$\big[I_1,\tilde{\mathfrak{h}}(E)\big]^\sim=\{0\}$
	      and $\big[\tilde{\mathfrak{h}}(E),I_1\big]^\sim\subset I_1$.
	 Next, by means of identity
	 (\ref{EqCompLeibnizHHM}) it can be shown by an easy computation that
	 $\{0\}=\big[\tilde{I}_1,E\big]^\sim
	    \stackrel{(\ref{EqDefTildeBracketEW2E})}{=}
	    -\big[E,\tilde{I}_1\big]^\sim$
	 which proves that $\tilde{I}_1$ is a two-sided ideal
	 of $\big(\tilde{\mathfrak{g}}(E),[~,~]^\sim\big)$.\\
	 Secondly, for the $K$-submodule $\tilde{I}_2$ identity (\ref{EqDefRLAToLieYCyclicRT}) yields
	$\{0\}=\big[\tilde{I}_2,E\big]^\sim
	\stackrel{(\ref{EqDefTildeBracketEW2E})}{=}
	-\big[E,\tilde{I}_2\big]^\sim$.
	Moreover identities (\ref{EqDefTildeBracketW2EW2E}) and (\ref{EqDefRLAToLieYCyclicRT}) imply
	$\big[\tilde{I}_2,\tilde{\mathfrak{h}}(E)\big]=
	    \{0\}$.
	On the other hand, thanks to the definition of the
	bracket (\ref{EqDefTildeBracketW2EW2E}) and to identity
	(\ref{EqDefRLAToLieYDeriROnT}) we get after a longer computation
	$\big[\tilde{\mathfrak{h}}(E),\tilde{I}_2\big]
	\subset \tilde{I}_2$
	whence $\tilde{I}_2$ is also a two-sided ideal of $\big(\tilde{\mathfrak{g}}(E),[~,~]^\sim\big)$.
	Finally, the sum of two two-sided ideals is always a
	two-sided ideal. Hence, $\tilde{I}(E)$ also is a two-sided ideal.\\
	\textit{ii.)} Let $\varpi: \tilde{\mathfrak{g}}(E)\to \mathfrak{g}(E)
	=\tilde{\mathfrak{g}}(E)/\tilde{I}(E)$ be the canonical projection which is
	a morphism of non-associative algebras. Hence, for all $\xi,\eta,\zeta\in
	\tilde{\mathfrak{g}}(E)$ it sends
	each term of the form $\mathbf{L}(\xi,\eta,\zeta)\big)$, see (\ref{EqDefLeibnizIdentityMapL}), to the corresponding term $\mathbf{L}\big(\varpi(\xi),\varpi(\eta),\varpi(\zeta)\big)$ in the factor algebra
	$\mathfrak{g}(E)$. On the other hand, by the definition (\ref{EqDefIdealOnePlusIdealTwo}) of $\tilde{I}(E)$ the term $\mathbf{L}(\xi,\eta,\zeta)\big)$ belongs to $\tilde{I}(E)=
	\mathrm{Ker}(\varpi)$
	thanks to the statements of Lemma \ref{LFELeibnizProperties}. This implies that
	$\mathbf{L}\big(\varpi(\xi),\varpi(\eta),\varpi(\zeta)\big)=0$, and
	the factor algebra  $\mathfrak{g}(E)$ thus satisfies the Leibniz identity.
	 Moreover, the bracket $[~,~]$ on 
	$\mathfrak{g}(E)$ induced by $[~,~]^\sim$ is antisymmetric (since $\tilde{I}_1\subset\tilde{I}(E)$)	
    whence the bracket on $\mathfrak{g}(E)$ is a Lie bracket.
    Obviously, since $\tilde{I}(E)\subset \tilde{\mathfrak{h}}(E)$ we can infer that $\mathfrak{h}(E)=\tilde{\mathfrak{h}}(E)/\tilde{I}(E)$
    	is a subalgebra of $\mathfrak{g}(E)$, and since $\tilde{I}(E)\cap E=\{0\}$ we have that the image of the subspace
    	$E$ modulo $\tilde{I}(E)$ is isomorphic to
    $E/(\tilde{I}(E)\cap E)=E/\{0\}\cong E$.  Finally, the definition of the
    	bracket in (\ref{EqDefTildeBracketEW2E}) shows that
    $E$ is invariant under the adjoint action of
    $\mathfrak{h}(E)$, and therefore 
    $(\mathfrak{g}(E),\mathfrak{h}(E),E)$ is a RLA pair.\\
    \textit{iii.)} The map $\tilde{\chi}$
    is a well-defined $K$-linear map thanks to the universal properties of the Grassmann algebra $\Lambda E$. The fact that $\tilde{\chi}$
    is a morphism of nonassociative algebras $\big(\tilde{\mathfrak{g}}(E),[~,~]^\sim\big)
    \to \big(\mathfrak{g}',[~,~]'\big)$, i.e.~for all $\xi_1,\xi_2\in \tilde{\mathfrak{g}}(E)$: 
    $\tilde{\chi}([\xi_1,\xi_2]^\sim) =[\tilde{\chi}(\xi_1),\tilde{\chi}(\xi_2)]'$,
    is a straight-forward computation
    using (\ref{EqDefTildeBracketEE}), (\ref{EqDefFofPsi}), and (\ref{EqDefRLAMapT}) for any $\xi_1=v_1,\xi_2=v_2\in E$, 
    (\ref{EqDefTildeBracketEW2E}), (\ref{EqDefRLAMapR}), and (\ref{EqDefFofPsi})
    for any $\xi_1=v\in E,\xi_2=v_1\wedge v_2\in \Lambda^2E$, and
     (\ref{EqDefTildeBracketEW2E}), (\ref{EqDefRLAMapR}), (\ref{EqDefFofPsi}),
    and the Jacobi identity for the Lie bracket $[~,~]'$ for all
    $v_1\wedge v_2,w_1\wedge w_2\in\Lambda^2E$.\\
    According to (\ref{EqDefIdealOne}) the ideal $\tilde{I}_1$ is spanned by
    $\big[v_1\wedge v_2, w_1\wedge w_2\big]^\sim
    +\big[w_1\wedge w_2, v_1\wedge v_2\big]^\sim$. Using the fact that $\tilde{\chi}$
    is a morphism of non-associative algebras and Definition (\ref{EqDefFofPsi}) we get
    \[
      \tilde{\chi}\Big(\big[v_1\wedge v_2, w_1\wedge w_2\big]^\sim\Big)
      =\big[\tilde{\chi}(v_1\wedge v_2),\tilde{\chi}(w_1\wedge w_2)\big]'
      =\Big[\big[\chi(v_1),\chi(v_2)\big]'_{\mathfrak{h}'},
         \big[\chi(w_1),\chi(w_2)\big]'_{\mathfrak{h}'}
            \Big]'.
    \]
    The right hand side of this equation is antisymmetric in the arguments
    $v_1\wedge v_2$ and $w_1\wedge w_2$ since $[~,~]'$ is a Lie bracket on $\mathfrak{g}'$. This implies that $\tilde{\chi}$ vanishes on $\tilde{I}_1$.
     Moreover, the fact that
     $\tilde{\chi}(\tilde{I}_2)=\{0\}$ is computed in a straight-forward manner using eqs (\ref{EqDefFofPsi}), (\ref{EqDefRLAMapT}), and again the Jacobi identity for the Lie bracket $[~,~]'$.
     It follows that $\tilde{\chi}$ maps the ideal $\tilde{I}(E)=\tilde{I}_1+\tilde{I}_2$
    of $\tilde{\mathfrak{g}}_{E}$ to $\{0\}$
    whence the map $\tilde{\chi}$ passes to the quotient to a well-defined morphism of Lie algebras
     $\check{\chi}:\mathfrak{g}(E)\to
    \mathfrak{g}'$ mapping the subalgebra
    $\mathfrak{h}(E)$ to the subalgebra $\mathfrak{h}'$
    and the $K$-submodule $E$ to the $K$-submodule $\mathfrak{m}'$.
    It follows that $\check{\chi}$ is a morphism of RLA pairs.
\end{proof}
\begin{corollary}\label{CDefFunctorF}
	The assignment 
	\begin{equation}\label{EqDefFunctorF}
		\mathbf{F}: \mathcal{LY} \longrightarrow  \mathrm{m}\mathcal{RLP},
	\end{equation}
	which associates to each Lie-Yamaguti algebra
	$(E,T,R)$ the reductive Lie algebra pair 
	$\mathbf{F}(E,T,R)=(\mathfrak{g}(E),\mathfrak{h}(E),E)$, 
	see (\ref{EqDefGE}) and (\ref{EqDefHE}), and to each
	morphism $\psi:(E,T,R)\to (E',T',R')$ the morphism
	$\mathbf{F}\psi=\check{\psi}:
	(\mathfrak{g}(E),\mathfrak{h}(E),E)\to
	(\mathfrak{g}(E'),\mathfrak{h}(E'),E')$, see
	 (\ref{EqDefFofPsi}), is a covariant functor.
\end{corollary}
\begin{proof}
	The fact that $\mathbf{F}(E,T,R)$ is a RLA pair is
	shown in statement \textit{ii.)} of the preceding
	Lemma \ref{LFunctorFComputations}. Moreover,
	set the RLA pair $(\mathfrak{g}',\mathfrak{h}',\mathfrak{g}')$ occurring in statement \textit{iii.)} of Lemma
	\ref{LFunctorFComputations} equal to
	$(\mathfrak{g}(E'),\mathfrak{h}(E'),E')$. Since it is obvious that its associated LY algebra $\mathbf{G}\mathfrak{g}(E')$ (with binary and ternary brackets defined by (\ref{EqDefRLAMapT} and (\ref{EqDefRLAMapR}))
	 is equal to the LY algebra $(E',T',R')$ it follows that the morphism $\mathbf{F}\psi=\check{\psi}$ is well-defined thanks to statement $iii)$ of lemma \ref{LFunctorFComputations}. The fact that $\mathbf{F}$ preserves composition of morphisms and maps identity morphisms to identity morphisms is a straight-forward consequence of the preceding results and the definitions.
	Hence $\mathbf{F}$ is a covariant functor.
\end{proof}
\noindent In the following theorem we shall show that the functor
$\mathbf{F}$ is a left adjoint of the more obvious functor $\mathbf{G}$, see
  (\ref{EqDefFunctorG}),
and thus makes the RLA pair $(\mathfrak{g}(E),\mathfrak{h}(E),E)$ a universal object
which may be called the \textit{free RLA pair generated by the LY algebra $(E,T,R)$}:
\begin{theorem}\label{TPrincipal}
  The functor $\mathbf{F}$ defined in Corollary \ref{CDefFunctorF}
   is a left adjoint to the functor $\mathbf{G}$, see
  (\ref{EqDefFunctorG}). The natural
  isomorphism of the adjunction
  \begin{equation}
   \nu_{E,\mathfrak{g}'}:\label{EqDefAdjugantKModules}
   \mathrm{Hom}_{\mathrm{m}\mathcal{RLP}}\big(\mathbf{F}(E,T,R),
     (\mathfrak{g}',\mathfrak{h}',\mathfrak{m}')\big)
     \to
    \mathrm{Hom}_\mathcal{LY}
       \big((E,T,R),\mathbf{G}(\mathfrak{g}',\mathfrak{h}',\mathfrak{m}')\big) 
  \end{equation}
  is defined by the restriction
  \begin{equation}\label{EqDefAdjugantElements}
   \nu_{E,\mathfrak{g}'}(\vartheta):=\vartheta|_{E}:E\to 
   \mathfrak{m}'.
  \end{equation}
  Moreover, the components of the unit of the adjunction 
  $\boldsymbol{\eta}_E:E\to \mathbf{GF}(E)$ are natural isomorphisms of
  Lie-Yamaguti algebras, and the components of the counit of the adjunction $\boldsymbol{\epsilon}_{\mathfrak{g}}:\mathbf{FG}(\mathfrak{g})
  \to\mathfrak{g}$ are \textbf{surjective} natural morphisms of 
  $\mathfrak{m}$-generated RLA pairs.
\end{theorem}
\begin{proof}
	Since each morphism of RLA pairs $\vartheta:\mathfrak{g}(E)\to \mathfrak{g}'$ maps
	$E$ to the submodule $\mathfrak{m}'$ of
	$\mathfrak{g}'$, the restriction is well-defined and
	a morphism of Lie-Yamaguti algebras. According to
	Proposition \ref{PLRAtoLieY}, it follows that 
	$\nu_{E,\mathfrak{g}'}$ is well-defined morphism.\\
	 We shall first prove naturality of $\nu$ in both of its arguments: given an arbitrary 
	 morphism $\zeta:(E_1,T_1,R_1)\to (E_2,T_2,R_2)$ of Lie-Yamaguti algebras and an arbitrary morphism $\phi':\mathfrak{g}'_1\to
	 \mathfrak{g}'_2$ of RLA pairs we have to prove the commutativity of the following two diagrams:
	 \begin{equation}\label{EqDiagrams}
	   \begin{tikzcd}
	    \mathrm{Hom}_{\mathrm{m}\mathcal{RLA}}
	    \big(\mathfrak{g}(E_1),\mathfrak{g}'\big)
	    \arrow[r, "\nu_{E_1,\mathfrak{g}'}"]
	    & \mathrm{Hom}_{\mathcal{LY}}(E_1,\mathfrak{m}')\\
	    \mathrm{Hom}_{\mathrm{m}\mathcal{RLA}}
	    \big(\mathfrak{g}(E_2),\mathfrak{g}'\big)
	    \arrow[r, "\nu_{E_2,\mathfrak{g}'}"]
	    \arrow[u, "(~)\circ F(\zeta)"]
	    & \mathrm{Hom}_{\mathcal{LY}}(E_2,\mathfrak{m}')
	     \arrow[u, "(~)\circ \zeta"']
	   \end{tikzcd}
	   ~~\mathrm{and}~~
	   \begin{tikzcd}
	   \mathrm{Hom}_{\mathrm{m}\mathcal{RLA}}
	   \big(\mathfrak{g}(E),\mathfrak{g}'_1\big)
	   \arrow[r, "\nu_{E,\mathfrak{g}_1'}"]
	   \arrow[d, "\phi'\circ (~)"']
	   & \mathrm{Hom}_{\mathcal{LY}}(E,\mathfrak{m}_1')
	   \arrow[d, "G(\phi')\circ(~)"]
	   \\
	   \mathrm{Hom}_{\mathrm{m}\mathcal{RLA}}
	   \big(\mathfrak{g}(E),\mathfrak{g}'_2\big)
	   \arrow[r, "\nu_{E,\mathfrak{g}_2'}"]
	   & \mathrm{Hom}_{\mathcal{LY}}(E,\mathfrak{m}_2')
	   \end{tikzcd}.
	 \end{equation}
	 Indeed,
	 for all morphisms $\vartheta_2:\mathfrak{g}(E_2)\to \mathfrak{g}'$ of RLA pairs we have 
	  \begin{eqnarray*}
	     \nu_{E_1,\mathfrak{g}'}
	     \big(\vartheta_2\circ (\mathbf{F}\zeta)\big) & = &
	     \big(\vartheta_2\circ (\mathbf{F}\zeta)\big)|_{E_1}
	     =\vartheta_2|_{E_2}\circ \zeta
	     =\big(\nu_{E_2,\mathfrak{g}'}(\vartheta_2)\big)
	       \circ \zeta,
	  \end{eqnarray*}
	  which shows that the left diagram in  (\ref{EqDiagrams}) commutes. Next,
	   for every morphism
	  $\vartheta_1:\mathfrak{g}(E)\to \mathfrak{g}'_1$
	  of RLA pairs we get
	  \begin{eqnarray*}
	   (\mathbf{G}\phi')\circ (\nu_{E,\mathfrak{g}'_1}(\vartheta_1))
	   & =&
	   (\mathbf{G}\phi')\circ \left(\vartheta_1|_{E}\right)
	   =
	   \big(\phi'\circ \vartheta_1\big)|_{E}
	   =
	   \nu_{E,\mathfrak{g}'_2}(\phi'\circ\vartheta_1),
	  \end{eqnarray*}
	  which shows that the right diagram in (\ref{EqDiagrams}) commutes.\\
	  In order to find an inverse of 
	  $\nu_{E,\mathfrak{g}'}$, we associate to each
	  morphism $\chi:E\to \mathfrak{m}'$ of Lie-Yamaguti algebras the morphism of RLA pairs
	  $\check{\chi}:\mathfrak{g}(E)\to \mathfrak{g}'$ defined in lemma \ref{LFunctorFComputations} induced by the map in 
	  (\ref{EqDefFofPsi}). We compute
	  for all morphisms of Lie-Yamaguti algebras
	  $\chi:E\to \mathfrak{m}'$, for all morphisms of
	  RLA-pairs $\vartheta:\mathfrak{g}(E)\to 
	  \mathfrak{g}'$, and for
	  all $v,v_1,v_2\in E$:
	  \[
	     \big( \nu_{E,\mathfrak{m}'}(\check{\chi})\big)(v)
	     =\check{\chi}|_E(v) =\chi(v)
	  \]
	  and 
	  \begin{eqnarray*}
	  \big( \nu_{E,\mathfrak{m}'}(\vartheta)\big)^\vee(v)
	    &=& \big(\nu_{E,\mathfrak{m}'}(\vartheta)\big)(v)=
	       \vartheta|_E(v) =\vartheta(v),\\
	  \big( \nu_{E,\mathfrak{m}'}(\vartheta)\big)^\vee
	   \big(v_1\wedge v_2~\mathrm{mod}~\tilde{I}(E)\big) 
	   & = &   \Big[\nu_{E,\mathfrak{m}'}(\vartheta)(v_1),
	   \nu_{E,\mathfrak{m}'}(\vartheta)(v_2)
	       \Big]'_{\mathfrak{h}'}
	       =\Big[\vartheta|_E(v_1),
	         \vartheta|_E(v_2)
	       \Big]'_{\mathfrak{h}'}\\
	    & = & \big[\vartheta(v_1),\vartheta(v_2)
	          \big]'_{\mathfrak{h}'}
	          =\vartheta\big([v_1,v_2]_{\mathfrak{h}(E)}
	          \big)
	    =  \vartheta\big(v_1\wedge v_2~\mathrm{mod}~\tilde{I}(E)\big),
	  \end{eqnarray*}
which shows that the inverse of $\nu_{E,\mathfrak{m}'}$ is the
map $\chi\mapsto \check{\chi}$, whence $\nu$ is a natural isomorphism, and thus $\mathbf{F}$ is a left adjoint functor of $\mathbf{G}$..\\
The unit of the adjunction 
$\boldsymbol{\eta}:I
\Rightarrow \mathbf{GF}$
is a natural transformation whose component
$$\boldsymbol{\eta}_E:E\to \mathbf{GF}E=\mathbf{G}\mathfrak{g}(E)=E$$
is given by 
$\boldsymbol{\eta}_E
   =\nu_{E,\mathfrak{g}(E)}(\mathrm{id}_{\mathfrak{g}(E)})
   =\mathrm{id}_{\mathfrak{g}(E)}|_E=\mathrm{id}_E$ which is an isomorphism.\\
The counit of the adjunction $\boldsymbol{\epsilon}:\mathbf{FG}
\Rightarrow I$ is a natural transformation whose component 
$$\boldsymbol{\epsilon}_\mathfrak{g}:
\mathfrak{g}(\mathfrak{m})\to \mathfrak{g}$$
is given by
$\boldsymbol{\epsilon}_\mathfrak{g}
   ={\nu_{\mathfrak{m},\mathfrak{g}}}^{-1}
   (\mathrm{id}_\mathfrak{m})
   =\mathrm{id}^\vee_\mathfrak{m}$.
Hence, we have that
\begin{equation}\label{EqCompCounitFirstAdjunction}
	\forall~z,z_1,z_2\in\mathfrak{m}:~~~
  \boldsymbol{\epsilon}_\mathfrak{g}(z)=z~~~\mathrm{and}~~~
  \boldsymbol{\epsilon}_\mathfrak{g}\big(z_1\wedge z_2~\mathrm{mod}~\tilde{I}_\mathfrak{m}\big)
  =[z_1,z_2]_\mathfrak{h}.
\end{equation}
Obviously, $\mathrm{Im}(\boldsymbol{\epsilon}_\mathfrak{g})=\mathfrak{m}
\oplus K\mathrm{Span}\{[z_1,z_2]_\mathfrak{h}~|~z_1,z_2\in\mathfrak{m}\}
=\mathfrak{i}(\mathfrak{m})=\mathfrak{g}$. 
\end{proof}

\noindent \textbf{Remark:} The construction shows that 
for a \textit{finite-dimensional} LY-algebra $(E,T,R)$
its free reductive Lie algebra pair
$\mathbf{F}(E,T,R)= (\mathfrak{g}(E),\mathfrak{h}(E),E)$ is also
finite-dimensional since $\Lambda^2 E$ is obviously
finite-dimensional.

\noindent In order to locate the enveloping Lie algebra $\hat{\mathfrak{g}}(E)$
of a LY algebra $(E,T,R)$ in a categorical manner, we specialize to the subcategories
$\mathcal{LY}\mathrm{s}$, $\mathrm{m}\mathcal{RLP}\mathrm{s}$, and
$\mathcal{RLP}\mathrm{s}$ of $\mathcal{LY}$, $\mathrm{m}\mathcal{RLP}$, and
$\mathcal{RLP}$, respectively, by restricting all morphism sets to those of all \textbf{surjective morphisms}. We get the following 
\begin{theorem}\label{TEnvLieAlgUniversal}
	The assignment $\hat{\mathfrak{g}}:\mathcal{LY}\mathrm{s}\to \mathrm{m}\mathcal{RLP}\mathrm{s}$ defined by the classical enveloping Lie algebra
	is a covariant functor which is \textbf{right adjoint} to the obvious restriction
	$\mathbf{G}_s:\mathrm{m}\mathcal{RLP}\mathrm{s}\to \mathcal{LY}\mathrm{s}$
	of the functor $\mathbf{G}$. The counit of this second adjunction is an isomorphism, and the unit $\boldsymbol{\eta}_{s\mathfrak{g}}$ is surjective.\\
	As a consequence, we have the following compositioin of canonical surjective morphisms 
	in $\mathrm{m}\mathcal{RLP}$ for every $\big(\mathfrak{g},\mathfrak{h},
	\mathfrak{m}\big)$ in $\mathrm{m}\mathcal{RLP}$
	\begin{equation}\label{EqSurjCanMorph}
		\begin{tikzcd}
		\mathfrak{g}(\mathfrak{m})
		  \arrow[r, twoheadrightarrow, "\boldsymbol{\epsilon}_\mathfrak{g}"]
		      &   \mathfrak{g}
		      	\arrow[r, twoheadrightarrow, "\boldsymbol{\eta}_{s\mathfrak{g}}"]
		      	&\hat{\mathfrak{g}}(\mathfrak{m})
		\end{tikzcd}
	\end{equation}
	(where $\boldsymbol{\epsilon}_\mathfrak{g}$ is given by (\ref{EqCompCounitFirstAdjunction}) and
	whose kernels both are central ideals. In the finite-dimensional case
	it follows that any of the three algebras in  (\ref{EqSurjCanMorph}) is
	solvable (nilpotent) iff $\hat{\mathfrak{g}}(\mathfrak{m})$ is solvable
	(nilpotent).
\end{theorem}
\begin{proof}
	For any surjective morphism $\psi:(E,T,R)\to (E',T',R')$ of Lie-Yamaguti algebras and for all $v,v_1,v_2\in E$
	we should like to define the action of the would-be functor $\hat{\mathfrak{g}}$ on $\psi$ by $\big(\hat{\mathfrak{g}}\psi\big)(v):=v$
	and, in view of Definition (\ref{EqDefHatHOfE}) of the subalgebra
	$\hat{\mathfrak{h}}(E)$, by
	\begin{equation}\label{EqDefFunctorGHatOnMorphHHatHHat}
	    \big(\hat{\mathfrak{g}}\psi\big)\left(\hat{R}(v_1,v_2)\right):=
	       \hat{R}'\big(\psi(v_1),\psi(v_2)\big).
	\end{equation}
	In order to show that this is well-defined, we use the fact that the subalgebra $\hat{\mathfrak{h}}(E)\cong \Lambda^2E/\mathrm{Ker}(\hat{R})$ 
	and show that the $K$-linear map $\Lambda^2\psi:\Lambda^2E\to\Lambda^2E'$
	maps the kernel of $\hat{R}$ to the kernel of $\hat{R}'$ so that
	$\Lambda^2\psi$ passes to the quotient: indeed, let
	$\alpha=\sum_{i=1}^Nv_i\wedge w_i$ be an arbitrary element of 
	$\mathrm{Ker}\big(\hat{R}\big)$, where $v_1,\ldots,v_N,w_1,\ldots,w_N\in E$. We compute for all $v\in E$
	\begin{eqnarray*}
	  \Big(\hat{R}'\big(\big(\Lambda^2\psi\big)(\alpha)\big)\Big)\big(\psi(v)\big) & = &
	     	\sum_{i=1}^NR'\big(\psi(v_i),\psi(w_i)\big)\big(\psi(v)\big)
	     	=
	     	\sum_{i=1}^N\psi\big(R(v_i,w_i)v\big)
	   =  \psi\left(\hat{R}(\alpha)(v)\right)
	  = 0.
	\end{eqnarray*}
	\textit{Since $\psi$ is surjective} we have that $\hat{R}'\big(\big(\Lambda^2\psi\big)(\alpha)\big)=0$ and thereby proving that
	$\big(\Lambda^2\psi\big)\big(\mathrm{Ker}\big(\hat{R}\big)\big)\subset \mathrm{Ker}\big(\hat{R}'\big)$ which makes
	formula (\ref{EqDefFunctorGHatOnMorphHHatHHat}) well-defined. Note that
	the linear map $\Lambda^2\psi$ is surjective whence $\hat{\mathfrak{g}}\psi$ is automatically surjective. It is a routine check that $\hat{\mathfrak{g}}\psi$ is a morphism of RLA pairs using 
	(\ref{EqDefFunctorGHatOnMorphHHatHHat}), (\ref{EqDefEnvelLieAlgBracket}), (\ref{EqDefRLAToLieYDeriROnR}) and (\ref{EqDefMorphLY}). Next, in order to prove the adjunction
	\[
	   \nu_{\mathfrak{g},E'}:
	   \mathrm{Hom}_{\mathcal{LY}\mathrm{s}}
	   \left(\mathbf{G}_\mathrm{s}(\mathfrak{g},\mathfrak{h},\mathfrak{m}),
	           (E',T',R')\right)
	           \to
	     \mathrm{Hom}_{\mathrm{m}\mathcal{RLP}\mathrm{s}}  
	     \left((\mathfrak{g},\mathfrak{h},\mathfrak{m}),
	          \hat{\mathfrak{g}}(E',T',R')\right),
	\]
	we set $(\nu_{\mathfrak{g},E'}(\chi))(z):=\chi(z)$, and 
	\[
	\forall~z_1,z_2\in 
	\mathfrak{m}:~~~~
	(\nu_{\mathfrak{g},E'}(\chi))\left([z_1,z_2]_\mathfrak{h}\right)
	    := -\hat{R}'\big(\chi(z_1),\chi(z_2)\big)
	\]
	which is shown to be well-defined in a completely analogous way as for 
	(\ref{EqDefFunctorGHatOnMorphHHatHHat}) by observing that the map
	$\Lambda^2\mathfrak{m}\to \mathfrak{h}:z_1\wedge z_2\mapsto [z_1,z_2]_\mathfrak{h}$ is surjective since $\mathfrak{g}$ is 
	$\mathfrak{m}$-generated, and that its kernel is mapped by $\Lambda^2\chi$
	to the kernel of $\hat{R}'$ thanks to the surjectivity of $\chi$. By a longer, but straight-forward computation, it can be checked that each $\nu_{\mathfrak{g},E'}(\chi)$ is a morphism of $\mathrm{m}\mathcal{RLP}\mathrm{s}$, that its restriction to $\mathfrak{m}$ is the inverse of $\nu_{\mathfrak{g},E'}$, and that it defines a natural isomorphism.
	Finally, it is easily shown using  (\ref{EqDefTildeBracketEE}),
	(\ref{EqDefTildeBracketEW2E}), and (\ref{EqDefTildeBracketW2EW2E}) that
	$\mathrm{Ker}(\hat{R})$ contains the two-sided ideal $\tilde{I}$ of
	$\tilde{\mathfrak{g}}(E)$ and that the quotient $\mathrm{Ker}(\hat{R})/
	\tilde{I}$ is central in $\mathfrak{g}(E)$ which completes the proof of the theorem.
\end{proof}

\begin{small}

\end{small}


\begin{thebibliography}{99}


    
\bibitem{BEM1} Benito, P.,  Elduque,  A.,   Mart\'{\i}n-Herce,  F.: \textit{Irreducible Lie-Yamaguti algebras.} J. Pure Appl. Algebra \textbf{213} (2009) 795--808.


\bibitem{BEM2} Benito, P.,  Elduque,  A.,   Mart\'{\i}n-Herce,  F.: \textit{Irreducible Lie-Yamaguti algebras of generic type.} J. Pure Appl. Algebra \textbf{215} (2011), 108--130.



      



\bibitem{CD}  Chapoton, F., Dotsenko, V.:
  \textit{Yamaguti algebras and noncrossing partitions.}
         {\tt arXiv:2510.03148}


\bibitem{F} Feldvoss, J.:
\textit{Leibniz algebras as non-associative algebras.}
Vojt\v{e}chovsk\'{y}, P. (ed.) et al., Nonassociative mathematics and its applications.
Fourth mile high conference on nonassociative mathematics, Denver, CO, USA, July 29 -- August 5,
2017. Proceedings. Providence, RI: American Mathematical Society (AMS). 
Contemp.~Math.~\textbf{721},
115--149 (2019). 


\bibitem{Hel78} Helgason, S.: \textit{Differential Geometry, Lie Groups and Symmetric Spaces}.
      Academic Press, New York, 1978.
       


\bibitem{Jac49} Jacobson, N.: \textit{Lie and Jordan Triple Systems}. 
Amer.~J.~Math.~\textbf{71} (1949) 149--170.
       

\bibitem{KW01} Kinyon, M.K., Weinstein, A.: \textit{Leibniz algebras, Courant algebroids, and multiplications on reductive homogeneous spaces}.
 Amer.~J.~Math.~\textbf{123} (2001), 525--550.


 \bibitem {KN63} Kobayashi, S., Nomizu, K.: 
  \textit{Fondations of Differential Geometry},
Vol I, Interscience Publishers, Wiley and Sons, New York, 1963.

 \bibitem{KN69} Kobayashi, S., Nomizu, K.: 
  \textit{Foundations of Differential
    Geometry.} Vol II, Interscience Publishers, Wiley, New York, 1969.

 \bibitem{Kos60} Kostant, B.: \textit{A characterization of invariant affine connections}. Nagoya Math. J.~\textbf{16} (1960), 35--50.
 
\bibitem{Lod92} Loday, J.-L.: \textit{Cyclic Homology}. Springer, Berlin, 1992.
     
\bibitem{Mac98} Mac Lane, S.: \textit{Categories for the Working Mathematician}. 
      2nd ed., Springer, New York, 1998.
      

      


\bibitem{NVH65} Nguyen-van Hai: \textit{Relations entre les diverses obstructions relatives \`{a} 
	l'existence d'une connexion lin\'{e}aire invariante sur un espace homog\`{e}ne.}
C.~R.~Acad.~Sci.~Paris \textbf{260} (1965), 45--48.

\bibitem{Nom54} Nomizu, K.: \textit{Invariant affine connections on homogeneous spaces}. Amer.~J.~Math.~\textbf{76} (1954), 33--65.
     
     

\bibitem{Smi11} Smirnov, O.N.: \textit{Imbedding of Lie triple systems into Lie algebras.} 
J.~Algebra~\textbf{341} (2011), 1--12.

\bibitem{Yam58} Yamaguti, K.: 
\textit{On the Lie triple system and its generalization},
J.~Sci.~Hiroshima Univ., Ser.~A, 
\textbf{21} (1958), 155--160.

\end{thebibliography}
\end{document}